\documentclass[11pt, reqno]{amsart}
\pdfoutput=1
\usepackage[utf8]{inputenc}
\usepackage{graphicx,color,amssymb, epsf, amsmath, geometry}
\usepackage{longtable,verbatim}
\usepackage{amsthm}
\usepackage[english]{babel}
\usepackage{subfig}
\def\al{\alpha}

\def\Real{\mathbb{R}}

\def\one{{\mathbf 1}}

\def\HH{{\mathcal H}}

\def\Prob{{\mathbb P}}

\def\rr{{\mathbf r}}
\def\xx{\textbf {\textit x}}
\def\HH{\textbf {\textit H}}

\def\ex{{\mathbb E}}

\def\separ{{\mathcal S}}
\def\vr{{\mathbf R}}
\def\vrinv{{\vr^{-1}}}
\def\prob{{\mathbb P}}

\def\phsp{{\mathcal P}}
\def\monot{{\mathcal M}}

\def\ba{\begin{pmatrix}}
\def\ea{\end{pmatrix}}
\def\nn{\nonumber}
\def\erfc{\rm erfc}
\def\rar{\rightarrow}

\newtheorem{dfn}{Definition}
\newtheorem{lem}{Lemma}
\newtheorem{thm}{Theorem}
\newtheorem{prp}{Proposition}
\newtheorem{rem}{Remark}

\newtheorem{cor}{Corollary}

\newtheorem{theo}{Theorem}

\title{Search on the Brink of Chaos}
\author{Y. Baryshnikov, V. Zharnitsky}

\address{ Department of Mathematics \\
University of Illinois  \\ Urbana, IL 61801 }
\date{\today}

\begin{document}

\begin{abstract}
The  Linear Search Problem is studied from the view point 
of Hamiltonian dynamics. For the specific, yet representative 
case of exponentially  distributed position of the hidden object, 
it is shown that the optimal orbit follows an unstable  
separatrix in the associated Hamiltonian system.  
\end{abstract}

\maketitle

\setcounter{tocdepth}{2}

\section{Introduction}
\label{sec-1}
The Linear Search Problem has a venerable history, going back to R.~Bellman ('63)
and A.~Beck ('64). They looked into
the following question:
\begin{quote}
  An object is placed at a point $\HH$ on the real line, according to a known
  probability distribution. A {\em search plan} (or {\em
    trajectory}) is a sequence $\xx=\{x_i\}_{i=1}^{\infty}$ with 
$\ldots  -x_4<-x_2<0<x_1<x_3<\ldots$ (or $\ldots -x_3<-x_1<0<x_2<x_4<\ldots$). 
  A search is performed by a {\em searcher} walking alternating to the
  points of the search plan, starting at $0$,
  until the point $\HH$ is found.
\end{quote}

The total distance traveled till the point is found is $L(\xx, \HH)$, and 
the {\em cost} of the search plan $\xx$ is given by
$$
E(\xx)=\ex \, [L(\xx, \HH)].
$$
The task is to find the plan $\xx$ minimizing $E(\xx)$. We are
therefore in the {\em average case} analysis situation.

The search problem has been also  studied in theoretical computer science, see {\em e.g.} \cite{reif}, where it was 
called cow-path problem.  There have been many interesting generalizations such as search on rays, rendezvous, search with turn cost etc. \cite{revenge,newman,alpern99}. Finally, there is  some recent work in connection  with robotics, see {\em e.g.} \cite{bretl}. 

\subsection{Background on Linear Search Problem}
\label{sec-1.1}
This Linear Search Problem was studied mostly by Anatole Beck and his coauthors
in a series of papers where they analyzed to great details the archetypal case of normally 
distributed $\HH$ (see \cite{franck, beck, rides,gal,lim}). 
It turned out that the
candidates for optimal trajectories form a 1-parametric family
(parameterized by the length of the first excursion $|x_1|$). Using 
careful analysis Beck  further reduced the choice of the candidates to just two
initial points, of which one turned out to be the best by
numerics. On the nature of these initial points, \cite{beck}
stated:
\begin{quote}
 ...we opine that this is a question whose answer will not shed much
 mathematical light. 
\end{quote}

This note aims at uncovering the underlying geometric structure of the Linear Search
Problem. Specifically, we argue that the correct framework here is
that of {\em Hamiltonian dynamics}, especially where hyperbolicity of
the underlying dynamics can be deployed. In our geometric picture  the mysterious two points
naturally appear at the intersection of a separatrix (that is present in the associated Hamiltonian system) with the curve of initial turning points.

To this end we analyze in
detail a one-sided version of the Linear Search Problem which we describe next. 
The original problem considered by Beck is addressed from the same
viewpoint in the appendix.  

We restrict our proofs mostly to the exponentially distributed
 position $\HH$: this is done primarily to keep the presentation succinct and
clear. In the appendix we demonstrate that our approach with small 
modifications works 
for some other distribution, {\em e.g.} for one-sided Gaussian. We believe that even more general classes
 of distributions can be also analyzed  - this will be done in a follow-up paper.

\subsection{Half-line problem}
\label{sec-1.2}
We concentrate here on a {\em one-sided gatherer} version of the search problem.
Here, the hiding object $\HH$ is located on the {\em half-line} $\Real_+$, according
to some (known) probability distribution. One searches for $\HH$
according to the {\em plan} 
\[
\xx=\{0=x_0<x_1<x_2<\ldots < x_k < \ldots \},
\]
 and stops after
the step $n=n(\xx,\HH)$ {\em iff} the point $\HH\in(x_{n-1}, x_n]$. 
One
can think of a {\em gatherer} who mindlessly collects anything on the 
way, bringing the loot to the origin, where the results are analyzed
(in a contrast to the  
{\em searcher}, who stops as soon as the sought after object is found).

As in the original version, 
one needs to minimize the average cost of the search, which in our case is given by
\begin{equation}\label{eq:cost}
E(\xx)=\ex [L(\xx,\HH)] =\ex \left [ \sum_{k=1}^{n(\xx,\HH)} x_k \right ].
\end{equation}

\begin{figure}[ht]
 \centering
 \includegraphics[width=.47\textwidth ]{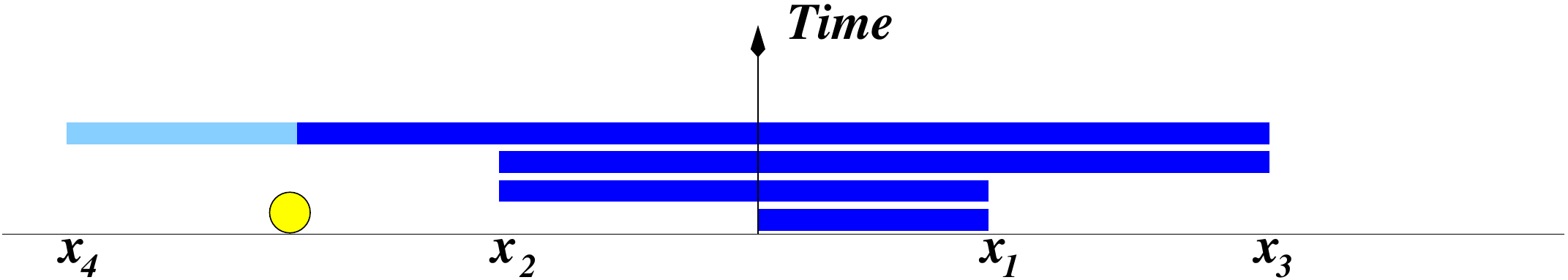}
 \includegraphics[width=.47\textwidth ]{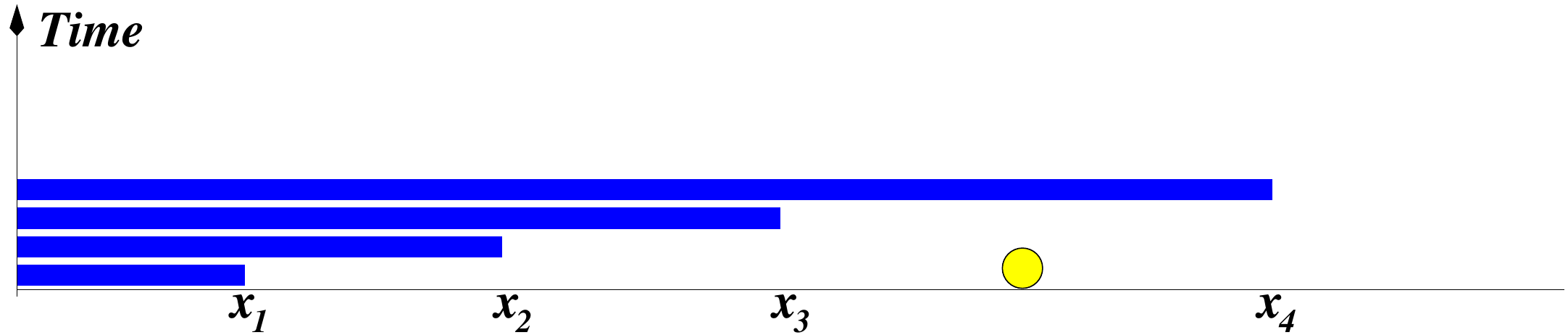}
 \caption{Linear Search Problem: two-sided searcher on the left,
   one-sided gatherer on the right. The cost of the indicated plans
   given the positions of the hidden objects are shown by darker shade.}
 \label{fig:lsp}
\end{figure}

{}

\subsection{Motivation}
One-sided linear search appears naturally in quite a few applications.
The initial motivation was the problem of search in unstructured Peer-to-Peer
storage systems, analyzed in \cite{yulik_oper}, where the relevance of
Hamiltonian dynamics was first noticed. In such an unstructured network,
one is sequentially flooding some (hop-)vicinity of a node, see Figure \ref{network},
 with request for an item, setting the Time-to-Live at some limit, until the item is
found. 
The cost of a plan is the total number of queries at all nodes of
the network, representing the per query overhead.  

\begin{figure}
 \includegraphics[width=.45\textwidth]{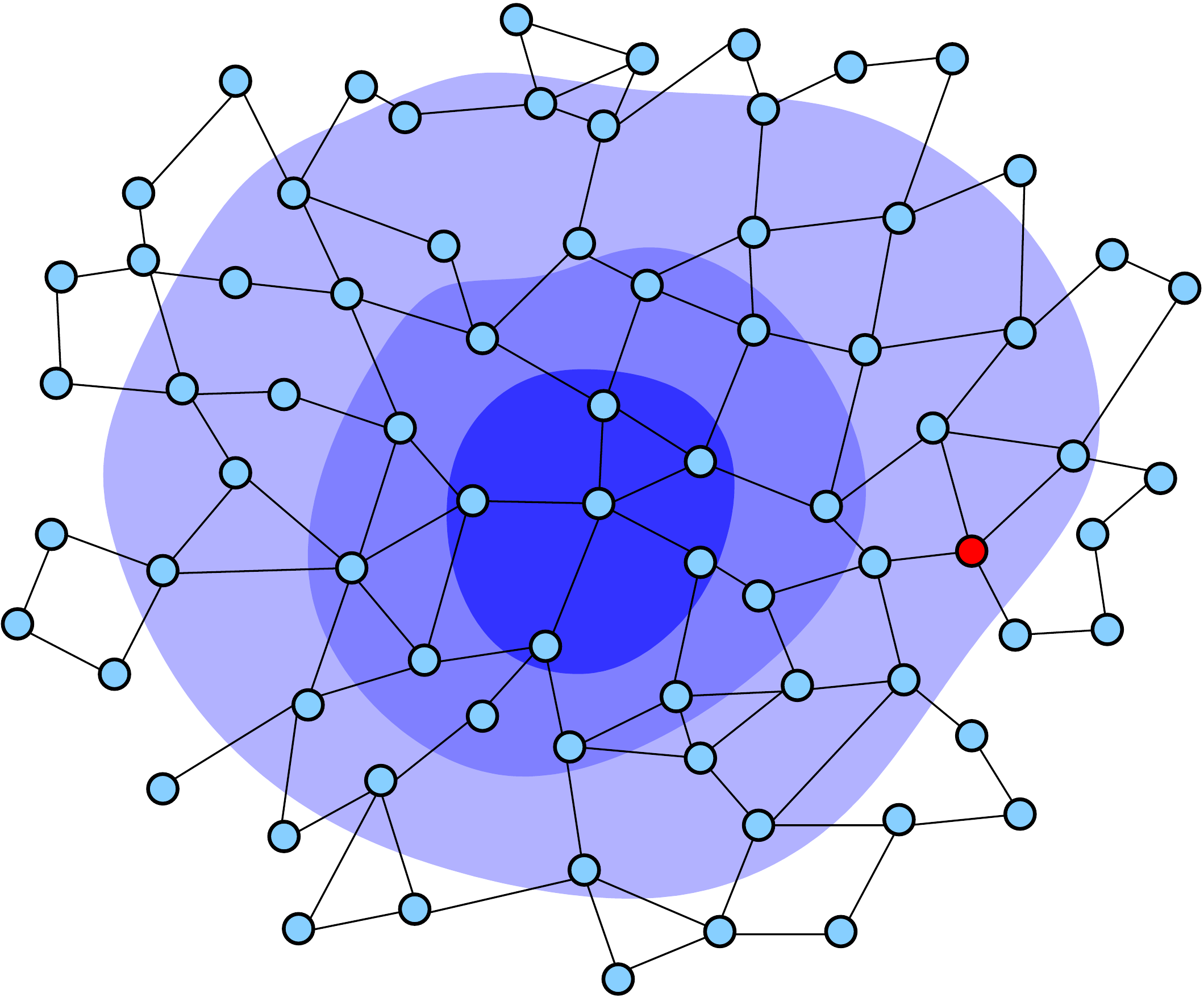}
\caption{Search for an object in Peer-to-Peer unstructured network. 
The object is found after the 3rd flooding.}
\label{network}
\end{figure}

Further applications include {\em robotic
  search}, where one deals with programming a robot of low sensing and
computational capabilities, unable to recognize the objects it collects.
Also the problem of efficient eradication of unwanted phenomena (say
irradiation of a tumor) can be mapped onto our model.

\subsection{Outline of the results}
\label{sec-1.3}
We start with the general discussion of the one-sided
search problem, showing in particular that the natural necessary
condition of optimality implies that the optimal plan should satisfy a
three-term recurrence, the {\em variational recursion} 
(a discrete analogue of the Euler-Lagrange
equations). 
This reduces the dimension of phase space, but also introduces  Hamiltonian dynamics. 

We analyze in details a ``self-similar'' case of homogeneous tail
distribution function, also called a Pareto distribution, and see that the phase 
space is split naturally into a chaotic and monotonicity regions, divided by a {\em separatrix}. 

Hamiltonian dynamics associated to the variational recursion is then studied. We
set up the stage for a general distribution, but mostly constrain our
proofs to the case of 
the {\em exponentially distributed position $\HH$ of the object},
i.e. to the case of 
$$
f(x):=\prob(\HH>x)=\exp(-x).
$$
We prove that the optimal trajectories should start at
the separatrix\footnote{This connection between energy minimizing orbits and 
invariant sets is reminiscent of the Aubry-Mather theory \cite{aubry}. 
There energy minimization is used to prove existence of the so-called 
 Aubry-Mather sets. Here we proceed in the other direction:  we establish an invariant set in order to find minimal ``energy'' orbits.  }. On the other hand, the plans satisfying variational recursion 
are represented by a one-dimensional curve. The intersection of the separatrix 
with the curve gives two candidates for the starting position,
mirroring the situation in the original setting of Beck {\em et al}'s papers.
We conclude with several open questions. 

Occasionally, we use several
standard notions  from the theory of dynamical systems; for definitions
we refer to \cite{katok}.

\section{Basic properties}
\label{sec-2}
\subsection{Basic notions}
\label{sec-2.1}
The input into the search algorithm is a {\em plan}, or a {\em trajectory}
\[
\xx=\{x_0=0, x_1, \ldots \, x_k, \ldots \}, x_k\geq 0, x_k \rar \infty,
\]
that is an unbounded sequence of turning points. 
Below we list some simple properties of the cost functional (\ref{eq:cost}):

\begin{prp}\label{prp:basic1}
The cost of a plan is given by
\begin{equation}\label{eq:cost-functional}
E(\xx)=\sum_{k=1}^{\infty} x_k \, \Prob( \HH >x_{k-1}) = 
\sum_{k=1}^{\infty} x_k f(x_{k-1}). 
\end{equation}
Any optimal search plan is strictly monotone.  In other words, if 
a plan $\xx = \{x_n\}$ is not strictly increasing, 
there is a naturally modified strictly monotone plan 
$\tilde \xx = \{ \tilde x_n\}$ such that $E(\tilde \xx) < E(\xx)$.
\end{prp}
\begin{proof}
The contribution to the average cost is the length of excursion times the probability that such excursion will have to occur:
$$
E(\xx,\HH)=\sum_k x_k \cdot \one( \HH>x_{k-1})
$$
which implies \eqref{eq:cost-functional}.

Now, assume that a plan ${\bf x}$ is not strictly monotone. Consider a modified plan $\tilde \xx$, where the turning points  preventing strict monotonicity are removed. Then, as can be verified by straightforward estimates, $E(\tilde \xx) < E(\xx)$.
\end{proof}

\begin{prp}\label{prp:basic2}
If the position of the object is {\em known}, then the cost of its
  recovery, \mbox{$L:=\ex [\HH]$},  is a lower bound on the cost of any
  trajectory   
\[
E({\bf x}) \geq L.
\]
There exists a plan of cost at most
  $4L+\epsilon$ (thus finite if $L$ is). 
\end{prp}
\begin{proof}
First, note that the sum 
\[
E({\bf x}) = \sum_{k=0}^{\infty} x_{k+1}f(x_k)
\]
is bounded below by the integral 
\[
\int_0^{\infty}f(x)dx = L.
\]

Next, observe that
\[
L = \ex [\HH] = -\int_0^{\infty} x \cdot f^{\prime}(x) dx =  \int_0^{\infty} f(x) dx,
\]
by definition and using integration by parts once.

Then, using monotonicity of $f$ we estimate this integral from below
\[
L = \int_0^{\infty} f(x) dx = \sum_{k=0}^{\infty} \int_{x_k}^{x_{k+1}} f(x) dx \geq
 \sum_{k=0}^{\infty} (x_{k+1}-x_k)f(x_{k+1}).
\]
Evaluating the expression on the right over the geometric  sequence $x_0 = 0, x_k = A\cdot 2^{k-1}$  $( k = 1, 2, \ldots)$, we have
\[
L \geq \frac 1 4  \sum_{k=0}^{\infty} f(x_{k+1})x_{k+2}. 
\]
Adding $x_1= A$ to both sides, we obtain
\[
4L + A \geq E({\bf x}),
\]
which proves the claim since $A$ can be taken arbitrarily small.
\end{proof}

If the tail distribution function is continuously differentiable (or even Lipshitz)
  on $[0,\infty)$, then
  the optimal trajectory does exist. In particular, one need not
  consider bi-infinite trajectories $\{ 0 < \ldots < x_{-2}< x_{-1}<x_1<x_2 < \ldots\}$. This is an extension of the  corresponding result  for the two-sided search, see {\em e.g.} \cite{beck_exist}. 
However, for
  completeness, we give an independent proof in the next section. The Lipshitz property 
is essential, as was also observed by Beck and Franck, since one can construct an example 
for which no sequence with finitely many terms near zero is optimal. In other words, there is no first turning point, see example in the next section.

\subsection{Variational recursion}

Optimality of a sequence implies a local condition.
\begin{prp} 
Assume the tail distribution function $f(x)=\Prob(\HH>x)$ is differentiable. 
 If the plan $\xx$ is optimal, then the terms $\{x_k\}$ satisfy the
{\em variational recursion}:
\begin{equation}\label{eq:var_rec}
f(x_{n-1})+x_{n+1} f'(x_{n})=0.
\end{equation}

\end{prp}
\begin{proof}
  It is immediate, if one notices that the cost depends on $x_k$ via
  only two terms, $f(x_{k-1}) x_k$ and $f(x_k) x_{k+1}$.
\end{proof}

This allows us to find $x_{n+1}$ as a function of $x_{n-1}, x_n$, 
$$
x_{n+1}=-\frac{f(x_{n-1})}{f'(x_n)}
$$
and to reconstruct the whole optimal plan from its first two points,
$x_0=0$ and $x_1$. 

In fact, it is useful to think of $\{x_k\}_{k=0,1,\ldots}$ as of
iterations of the mapping  
$\vr:\Real_+^2\to\Real_+^2$ given by 
$$
\vr:(x,y)\mapsto (y,-f(x)/f'(y))
$$
(which we will still be referring to as {\em
  variational recursion}).

\section{Existence of an optimal sequence}
\label{sec:exist_min}
For the two-sided (Beck-Bellman) search problem, the existence of 
the optimal search plans was shown in 
\cite{franck,beck_exist} and some improvements appeared in 
 the subsequent papers.
For completeness, we supply the existence proof for the  one-sided case,
as we consider in detail the associated nonlinear map. 

Recall the cost functional
\[
E({\bf x}) = \sum_{k=0}^{\infty} x_{k+1}f(x_k)
\]
and formulate the minimization problem:
\begin{eqnarray}\label{eq:minprob}
\hspace{10mm} E_0 = {\rm inf} \left \{ E({\bf x}),   {\bf x} =( \ldots ,x_{-2}, x_{-1},  x_0, x_1, x_2, ..., x_k, ...), x_j > 0,  j \in {\mathbb N}, x_k \rightarrow \infty  \right \}.
\end{eqnarray}
Note that we do not restrict the sequence to have the first term.  
We will prove this. On the other hand, if $f(x)$ does not vanish for any $x\geq 0$ 
there can be no other density points for an optimal plan, for otherwise 
the cost would be infinite.

Clearly,  $E_0 \geq 0$, since $E\geq 0$.
By definition of the infimum, there exists a minimizing sequence
$\{ {\bf x}^{(n)}\}$ such that
\[
E({\bf x}^{(n)}) \rar E_0.
\]
The goal is to show that there is a convergent subsequence such that
${\bf x}^{(n_k)} \rar {\bf x}^{*}$ and $E({\bf x}^{(n_k)})\rar E_0.$

\begin{prp}[Properties of minimizing sequences]
Assume $f$ is Lipshitz and $f(x)\neq 0$ for any $x\geq 0$.
In the minimization problem \eqref{eq:minprob},
there exist two positive monotone sequences, $\{ a_k \}_{k=0}^{\infty}, 
\{ b_k \}_{k=0}^{\infty}$, such that
$a_k < b_k$, $a_k \rightarrow \infty$, $b_k \rightarrow \infty$ and
there is a minimizing subsequence $\{ \xx^{(n)}\}$ such that $a_k < x_k^{(n)} < b_k $.
\end{prp}

\begin{proof}
First, we note that $E_0 \leq 4 L$ is a bounded quantity, see the previous section.
To prove existence of $\{ b_k \}$,  we first observe that any minimizing sequence  must satisfy
$E({\bf x}^{(n)}) \leq 2E_0$, for sufficiently large $n$. Thus, $x_1 \leq 2E_0 = b_1$ and 
then $x_2 f(x_1) \leq 2E_0$. Therefore, 
\[
x_2 \leq \frac{2E_0}{f(x_1)} \leq \frac{2E_0}{f(2E_0)} .
\]
We define then $b_2 = 2E_0/f(2E_0)$. Proceeding by induction,
\begin{equation}
b_{k+1} = 2E_0/f(b_k(E_0)),
\label{eq:beb}
\end{equation}
we obtain the desired sequence. Note that the sequence is strictly monotone as 
\[
x f(x) <  L \leq E_0 < 2E_0,
\]
and therefore, the mapping \eqref{eq:beb} cannot a fixed point $x=2E_0/f(x)$.

Thus, the sequence $\{ b_k \}$ monotonically grows to infinity and it bounds the
corresponding terms of the minimizing sequence.

To establish lower bounding sequence we prove\footnote{We use the notation 
$C_L$ for the Lipshitz constant.} 
\begin{lem}
Assume $f$ is Lipshitz and let ${\bf x}$ be a monotone, possibly bi-infinite,  sequence of turning points. Assume $x_m < 1/2C_L$, then the modified sequence $\tilde \xx$ with all $x_j (j < m)$ removed, will have lower cost.  
\end{lem}
\begin{proof}
Rewrite
\[
E(\xx) = \sum_k x_{k+1} f(x_k) = \ldots + x_{m-1} f(x_{m-2}) + x_m f(x_{m-1}) + x_{m+1} f(x_m) + \ldots
\]
and the modified sequence
\[
E(\tilde \xx) = x_m + x_{m+1}f(x_m) + \ldots.
\]
We need to show 
\[
x_m < \ldots +  x_{m-1} f(x_{m-2}) + x_m f(x_{m-1}).
\]
Rearranging some terms we get,
\[
\frac{1-f(x_{m-1})}{x_{m-1}} < \ldots + \frac{f(x_{m-2})}{x_m}.
\]
The left handside is bounded by the Lipshitz constant $C_L$ and the right handside 
is bounded from below by $1/x_m - C_L$. Therefore, by choosing 
$x_m< 1/2C_L$, we obtain the desired result. 
\end{proof}
Therefore, an optimal sequence of turning points is one-sided and there is at 
most one point in the interval $[0,\delta=1/2C_L]$. Then, we let $a_0 = 0$ and $a_1 = \delta$. 

Now, the sequence $\{a_k\}$ can be constructed using monotonicity 
$a_{k+1}\geq a_k$ and that there are finitely many terms on any interval 
of, say, unit size: $|\delta, \delta+1|,|\delta+1, \delta+2| $, etc. \\

Monotonicity has been proved in the previous section by showing that in nonmonotone sequence,
by deleting the appropriate terms, we obtain a strictly monotone sequence with smaller cost.

\end{proof}

\begin{theo}
There exists a converging subsequence,
 ${\bf x}^{(n)}\rightarrow {\bf x}^{(*)}$, where ${\bf x}^{(*)}$ is strictly monotone and  $x^{(*)}_k \rightarrow \infty$. The cost function converges $E({\bf x}^{(n)}) \rightarrow E({\bf x}^{*})=E_0$.
\end{theo}
\begin{proof}
Fix $N>0$ and let ${\mathbb P}_N \xx = (x_1, x_2, ..., x_N)$. For the minimizing sequence $\xx^{(n)}$, let $\xx^{(n_1)}$ be a subsequence for which
$x^{n_1}_1 \rightarrow x_1^*$. Take a subsequence of this subsequence, so that
${\mathbb P}_2 {\bf x}^{(n_2)} \rightarrow \{ x_1^*, x_2^* \} $. Proceeding further and using diagonal subsequence ${\bf x}^{(n_k),k}$, we obtain a convergent subsequence, which we will still denote by  
${\bf x}^{(n)}\rightarrow {\bf x}^*$. The limit  $\xx^*$ is a monotone sequence by 
construction. It must be also strictly monotone, for if not, {\em i.e.} if 
some terms are equal, we already know from the previous section that by removing repeated terms 
the cost is decreased, which contradicts the sequence being minimizing.

Now, to prove the second part of the theorem, let $E^N({\bf x})$ denote $N-th$ 
partial sum. 
Fix $N>0$ to be sufficiently large,
and observe that $E^N ({\bf x}^{(n)}) \rightarrow E^{N}({\bf x}^*)$ just by 
continuity. Because of the lower bounding sequence $\{ a_k\}$, we can take $N$ so large that  $x^{(n)}_N$ and $x_N^*$ are larger than any fixed number. Consider now the remainders
\[
E({\bf x}^{(n)}) - E^N ({\bf x}^{(n)}), \,\,\,\, E({\bf x}^*) - E^N({\bf x}^*),
\]
which are arbitrarily small. Indeed,
\[
E({\bf x}) - E^N ({\bf x}) = x_{N+1} f(x_N) +  x_{N+2} f(x_{N+1})+\ldots
\]
 and since the sequence is minimizing 
we can estimate the reminder by choosing, e.g. $x_{N+k}= 2^{k+1} x_{N-1}$. 
Next, using an argument similar to the one used in 
Proposition \ref{prp:basic2}, we obtain the bound
\[
E({\bf x}) - E^N ({\bf x}) \leq 4 \int_{x_{N-1}}^{\infty} f(x) dx.
\]
The same bound holds for the other reminder.  Thus, taking $x_{N-1}$ 
large enough we can assure the reminders to be arbitrarily small. 
This implies the convergence  $E({\bf x}^{(n)}) \rightarrow E({\bf x}^{*})=E_0$.
\end{proof}

Next we demonstrate that the Lipshitz condition is necessary. Indeed,  
without it we can construct an example with no initial turning point: \\ \\
\noindent
{\bf Example with singularity}.
If the tail distribution function is not Lipshitz then the sequence may fail to have the first turning point.  Here, we present a simple  example of one-sided search. 

Let $f(x)=1-\sqrt{x}$ and assume  the search is done on the unit interval 
$[0,1]$. It is also possible to modify this example to the infinite ray 
$(0,\infty)$ by changing $f(x)$ outside of any neighborhood of $0$ so it 
does not vanish anywhere.

Suppose, the optimal sequence is given by a one-sided 
sequence $\{ 0< x_1 < x_2 < x_3 < \ldots\}$ with the cost 
\[
E({\bf x}) = x_1 + x_2 (1-\sqrt{x_1}) + x_3 (1-\sqrt{x_2}) + \ldots.
\]

Let us insert another point $x_0: 0 < x_0 < x_1$, then the cost of modified sequence
is given by
\[
E({\bf \tilde x}) = x_0 + x_1 (1-\sqrt{x_0}) + x_2 (1-\sqrt{x_1}) + \ldots.
\]
Comparing them, we find that the cost of modified sequence is lower if and only if
\[
x_0 +x_1 (1-\sqrt{x_0}) < x_1  \Leftrightarrow	\sqrt{x_0} < x_1  \Leftrightarrow x_0 < x_1^2.
\] 
The latter inequality can be always achieved. Therefore, the optimal sequence does not have an initial turning point.

\section{Pareto distribution}
\label{sec-2.2}
In this section we present an explicit example which illustrates our general approach: 
the optimal plan of the search problem belongs to an invariant manifold (separatrix) of 
the associated Hamiltonian  map.
\subsection{Cost functional}
Consider a Pareto type tail distribution (analogous to that of \cite{reif})
\begin{eqnarray}
f(x) &=& x^{-\alpha} \,\,\, {\rm if} \,\,\ x \geq 1, \nonumber \\ 
f(x) &=& 1 \,\,\, \,\,\,\,\,\, \,\, {\rm if} \,\,\  0< x < 1,     \nonumber 
\end{eqnarray}
where we assume that $\alpha > 1$ in order to have a bounded expected value. 

We will use the notation, exceptionally, $x_0=1$, which makes  formulas look simpler.
Note that $x_0=1$ does not correspond to an actual turning point.
The expected cost is given by
\[
E({\bf x}) =  x_1 + f(x_1) x_2 + f(x_2) x_3 + ... = \sum_{n=0}^{\infty} \frac{x_{n+1}}{x_n^\al}.
\]

The variational recursion reads in this case
\[
x_{k+1} = \frac{1}{\al}\frac{x_k^{\al+1}}{x_{k-1}^{\al}}
\]
or equivalently
\[
\frac{x_{k+1}}{x_k^{\al}} = \frac{1}{\al}\frac{x_{k}}{x_{k-1}^{\al}} = \frac{1}{\al^k}\frac{x_{1}}{x_{0}^{\al}}= \frac{1}{\al^k} x_1.
\]
Therefore, for the sequences generated by the variational recursion,
with $x_1=x$, we can immediately compute the cost 
\[
E(\xx) = \sum_{n=0}^{\infty} \al^{-n} x_1 = x_1 \frac{\al}{\al -1},
\]
as a function of the initial condition $x_1=x$.

This expression indicates that $x_1$ should be as small as possible, provided 
the sequence satisfies the constraints of monotonicity and unbounded growth.

From the sequence definition, we have
\[
\frac{x_{k+1}}{x_k}= \frac{1}{\al} \left ( \frac{x_k}{x_{k-1}}\right )^{\al}
\]
or denoting the ratios by $r_k=x_k/x_{k-1}$,
\[
r_{k+1}= \al^{-1}r_k^{\al}, \,\,\, r_1 = x_1.
\]

Defining
$w_k=r_k \al^{-\frac{1}{\al-1}}$ gives
\[
w_{k+1} = w_k^{\al}.
\]
We clearly need  to take $w_1 \geq 1$, so that 
the ratios would not go to zero and the sequence $x_k$ would be monotone.
However, since we need $x_1$ to be as small as possible, we take $w_1
= 1$, resulting in $x_1 = r_1 = \al^{\frac{1}{\al-1}}$. Therefore, the minimal cost is given by 
\begin{eqnarray}\label{eq:excost}
E_0 = \frac{\al\cdot \al^{\frac{1}{\al-1}}}{\al -1}= \frac{\al^{\frac{\al}{\al-1}}}{\al -1},
\end{eqnarray}
and the optimal sequence is given by
\[
x_k = \al^{\frac{k}{\al-1}}.
\]
In a particular case of $\alpha=2$, the optimal sequence is given by 
 geometric series $x_k = 2^k$.

\subsection{Hamiltonian dynamics}
The global structure of the dynamics defined by the variational
recurrence in this homogeneous problem is shown on the Figure
\ref{fig:homo}. Here we draw the invariant curves for the trajectories
defined by $\vr$: the iterations of a point $(x_k, x_{k+1})$ found on one
of these curves, stays on it forever. The red (thick) line corresponds
to the optimal trajectory.

\vspace{5mm}

\begin{figure}[htb]
  \centering
  \includegraphics[width=.5\textwidth]{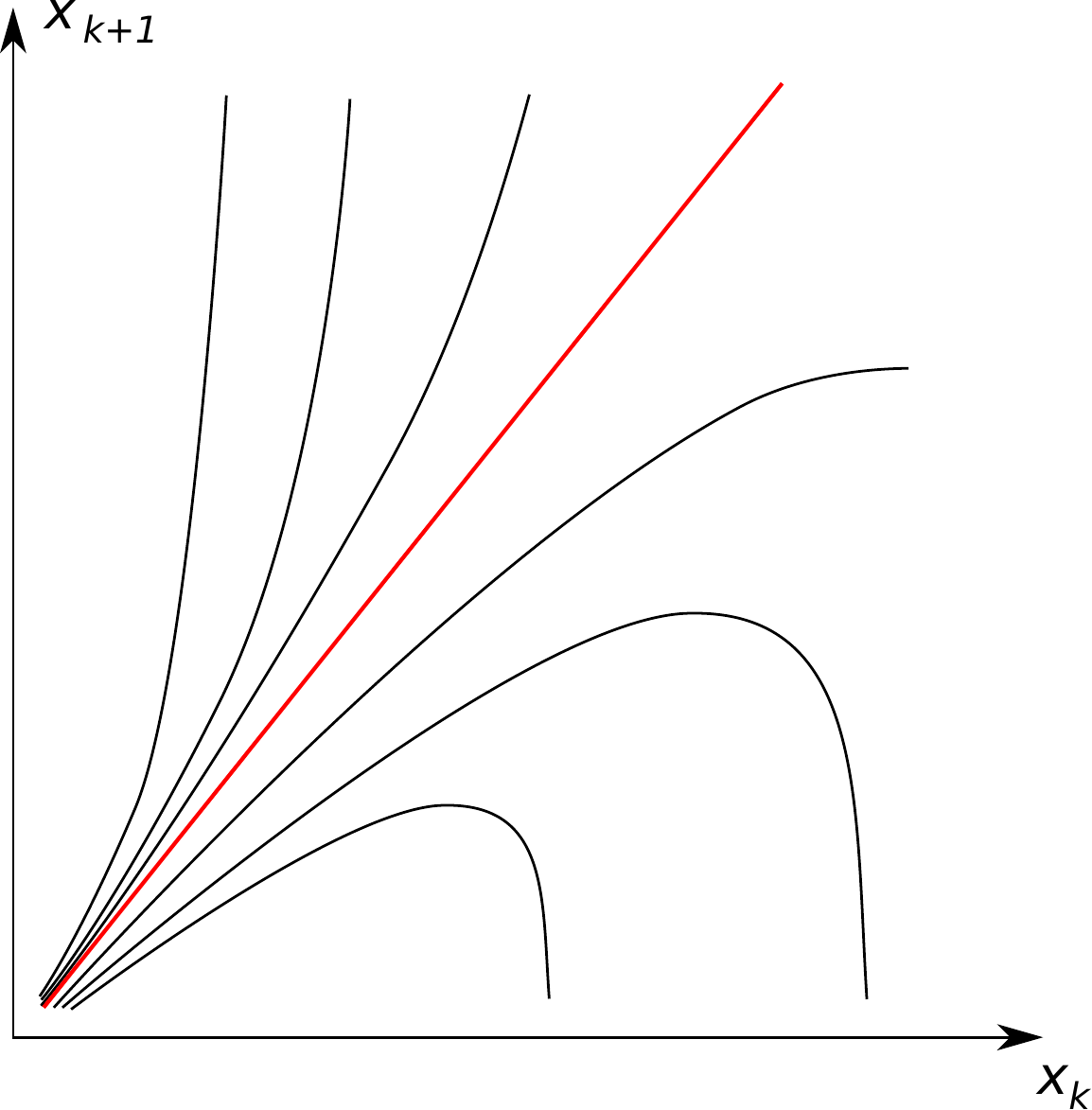}
  \caption{Phase portrait for the variational recursion for the homogeneous 
distribution. There are two regions: above the line
$x_{k+1}=\alpha^{1/(\alpha-1)}x_k$, 
where  all the orbits monotonically grow and below, where all the orbits 
lose monotonicity eventually.}
  \label{fig:homo}
\end{figure}

\vspace{5mm}
 
The qualitative dynamics in this case can be summarized as
follows:

\begin{itemize}
\item
There is a region of initial values $x_1$ where the variational recursion
stops making sense: the iterates become non-monotone. 
We will call this region {\em chaotic}\footnote{Albeit the dynamics is not really
chaotic in this particular case, we will see that this is rather
an exception.}. 
\item 
The optimal initial value is on the boundary of the chaotic set. 
\item
The growth of the optimal plan (exponential) is far slower than the
growth for generic initial values outside the chaotic region (where it
is super-exponential).
\end{itemize}

The sequences can be represented as solutions of the two dimensional nonlinear map
\begin{align*}
x_{k+1} &= r_{k+1}x_k \\
r_{k+1} &= \frac{1}{\al} r_k^{\al}.
\end{align*}

The ray $r=r^* = \al^{\frac{1}{\al-1}}$ is invariant.
Above this ray $r=r^*$, the orbits go rapidly to infinity. The orbits below $r=r^*$
are not monotone, because $r_k$ monotonically decreases to zero and while $x_k$ may
grow at first but after $r_k$ becomes less than 1, $x_k$ will be decreasing.

\section{Exponential tail distribution}
In this section we analyze in detail the prototypical case of exponential 
distribution. While, this case is sufficiently simple to allow complete
understanding, the Hamiltonian dynamics is no longer integrable. Therefore, 
the methods that we develop would apply to other cases of interest.

\subsection{Variational recursion}

We consider now several key properties of the variational recursion 
$\vr:(x,y)\mapsto (y,-f(x)/f'(y))$.

One of the basic observation is that it preserves an area form:
\begin{prp}
 The mapping $\vr$ preserves the area form $\omega=f'(x) dx\wedge dy$.
\end{prp}

This is a rather general fact: for any recursion obtained 
by extremization of the functional 
\[
E(\xx) = \sum_{k=0}^{\infty} F(x_k, x_{k+1}),
\]
the 2-form $\frac{\partial^2 F}{\partial x \partial y} dx\wedge dy$ is
invariant with respect to the associated two-dimensional  mapping.

It is possible to explicitly give the coordinates in which the 
variational recursion $\vr$ is {\em Hamiltonian}: if we use $(s,y)$, where
$s=f(x)$
in
lieu of $(x,y)$, then 
$$
\vr:(s,y)\mapsto (f(y),-s/f'(y));
$$
it maps $[0,1]\times\Real_+$ into itself and preserves the {\em Lebesgue area} $ds\wedge dy$.
We will be referring to these coordinate system as {\em standard}.

In the standard coordinates, the variational recursion for the
exponentially distributed $\HH$ (i.e. for $f(x)=\exp(-x)$) is given by
\[
\vr:(s,y)\mapsto (e^{-y},se^y).
\]

Further, one can see that $\vr$ has a unique stationary point,
$s=e^{-1}, y=1$. One can verify that this fixed point is elliptic.

\begin{figure}[htb]
\begin{center}
\includegraphics[height=3in]{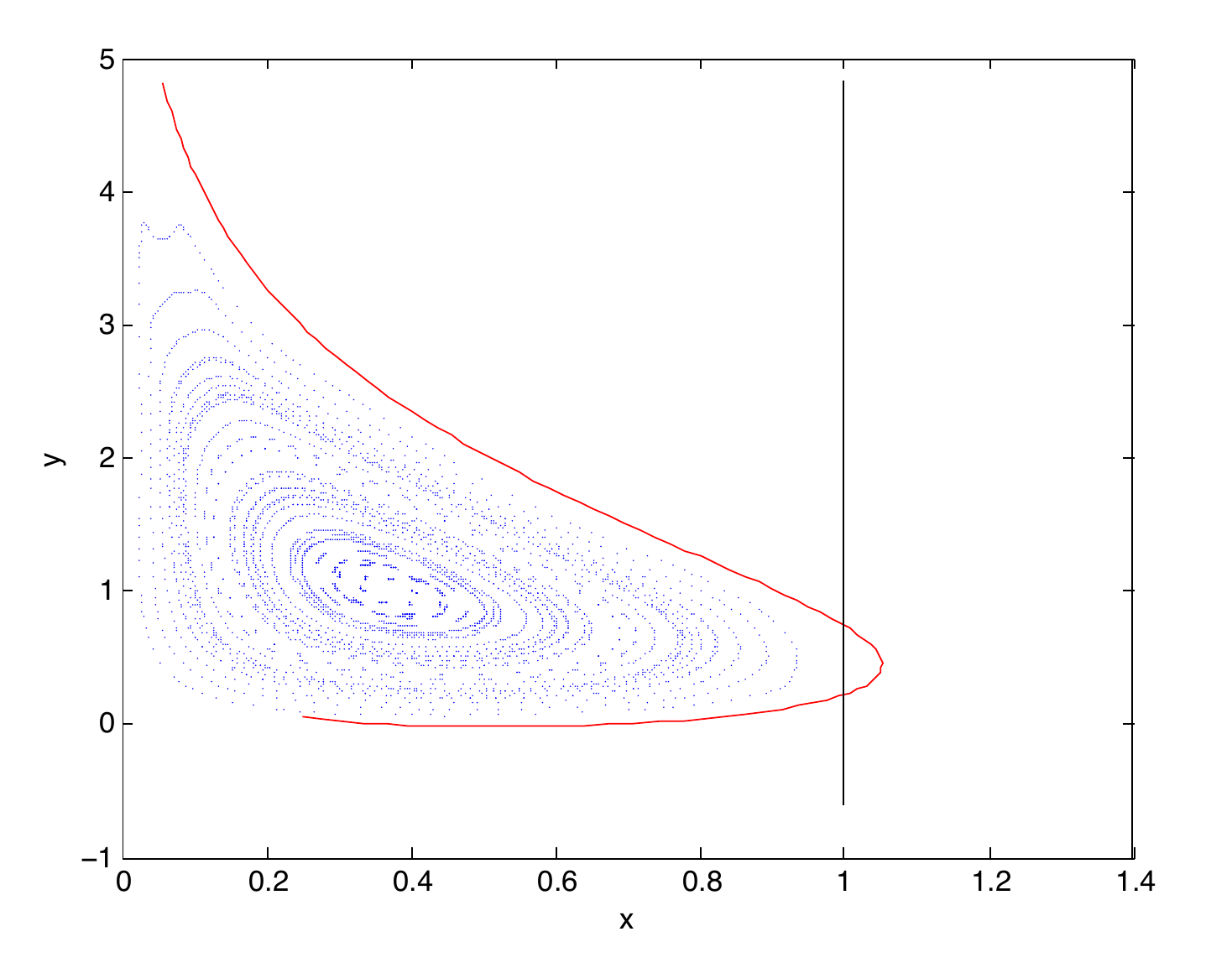}
\caption{Several orbits of the variational recursion for exponential
  distribution. The solid curve separates chaotic region from the monotonicity
  region. The region of interest is located to the left of the vertical line $x=1$. 
The monotone orbits outside of the chaotic region are not present as they 
are rapidly mapped to infinity. }
\label{fig:expo}
\end{center}
\end{figure}

\subsection{Cost functional and cost function}
\label{sec-2.4}
We already know that the optimal plan can be found only among the trajectories
satisfying the variational recursion. We will set $x_0=0$; under this
assumption the trajectories (not necessarily increasing) satisfying
the variational recursion  are parameterized by the first non-zero
term $x_1:=x$. We will be denoting the corresponding family of
trajectories as $\xx_\vr(x)=\{x_0=0, x_1=x, x_2=x_2(x),\ldots\}$. For
the exponentially distributed $\HH$, the first few terms of the family
$\xx_\vr(x)$ are given by $x_1=x; x_2=e^x; x_3=e^{e^x-x}$ and so on. \\

\noindent
{\bf Notation:}
We will use the term {\em cost functional} for
(\ref{eq:cost-functional}), 
defined on the space of all
trajectories $\xx$, while reserving  the term {\em cost function} for the
restriction of the
functional $E$ to the one-parametric curve $\xx_\vr(x)$ of solutions to variational
recursion, denoting the cost function by $E(x):=E(\xx_\vr(x))$. \\

For exponentially distributed $\HH$, the cost {\em function} is finite
on monotonic trajectories. Indeed, in this case, unless growing
without bound, the trajectory should converge to the only fixed point of
the variational recursion, which is impossible as it is an elliptic
point. If for some $K$, $x_K>1$, then for $k>K$, 
$$
x_{k+1}-x_k=\ln{x_{k+2}}\geq \ln{x_K}>0,
$$
and $x_k$ grows at least as an arithmetic progression, implying the
convergence of 
\[
E(\xx)=\sum_{k=0}^{\infty} x_{k+1}\exp(-x_k)=\sum_{k=0}^{\infty}\exp(-x_{k-1}).
\]

Now, as the cost function $E(x)$ is a function of one variable, and
we established that the optimal trajectory should be one of the family
$\xx_\vr(x)$, it might appear  that the rest is straightforward: to find the minimum of
$E(x)$ over the starting point $x_1=x$. However, if we take the formal
derivative
$$
\frac{dE}{dx}=\sum_{k=0}^{\infty} \frac{d}{dx}\left ( x_{k+1}(x)f(x_k(x) \right ),
$$ 
we will see that all the terms vanish, identically (precisely because
$\xx_\vr(x)=\{x_1(x),x_2(x),\ldots\}$ satisfies the variational
recursion). It might appear that $E(x)$ should be a constant! However, we
already computed $E(x)$ in an example in section \ref{sec-2.2}, and know that this is not the case.

The reason for this calamity is, of course, the fallacious differentiation of an
infinite sum of differentiable functions with wildly growing $C^1$ norms.

However, if we consider the {\em approximants}
$$
E^K(x)=\sum_{k=0}^K x_{k+1} f(x_k),
$$ 

they can be differentiated term by term, yielding
\begin{equation}\label{eq:diff_CK}
\frac{dE^K}{dx}(x)=x_{K+1}(x)f(x_K(x))
\end{equation}
(by telescoping).

As $E^K(x)$ approximates $E(x)$ to within $4 E_0 \, f(x_K)$, which uniformly converges to 
zero, the existence of a local minimum of $E(x)$ in an interval where $E$ is
finite would imply that the approximants 
$E^K(x)$ have local minima in that interval,
for all large enough $K$. Later we will use this observation to prove that the reduced cost function 
has optimal solution on the separatrix.

\section{Hamiltonian dynamics}
\label{sec-3}
Denote by $\phsp=\{1\geq s\geq 0, y\geq 0\}$
the phase space (in standard coordinates) 
on which the variational recursion
acts.
\subsection{Chaotic and monotone regions}
\label{sec-3.1}
\begin{dfn}
  The region $\monot_k$ of $k$-step monotonicity is defined as collection of
  points in $\phsp$ such that $k$-fold application of the $\vr$
  produces a monotonic (along $y$ coordinate) sequence. The
  intersection of all $\monot_k$ is denoted by $\monot_\infty:=\cap_k
  \monot_k$ and is called the {\em  region of monotonicity}. Its complement
  is called the {\em chaotic region}.
\end{dfn}

The boundary $\separ$ of the monotonicity  region is called the {\em
  separatrix}. It is not immediate that the separatrix is a {\em
  curve}: the monotone and chaotic regions can have rather wild
structure. However, we will see that the separatrix is indeed a smooth 
curve, and the relevant part of it can be represented as a graph of a function
in some appropriate coordinates.

\begin{figure}

\includegraphics[width=.6\textwidth]{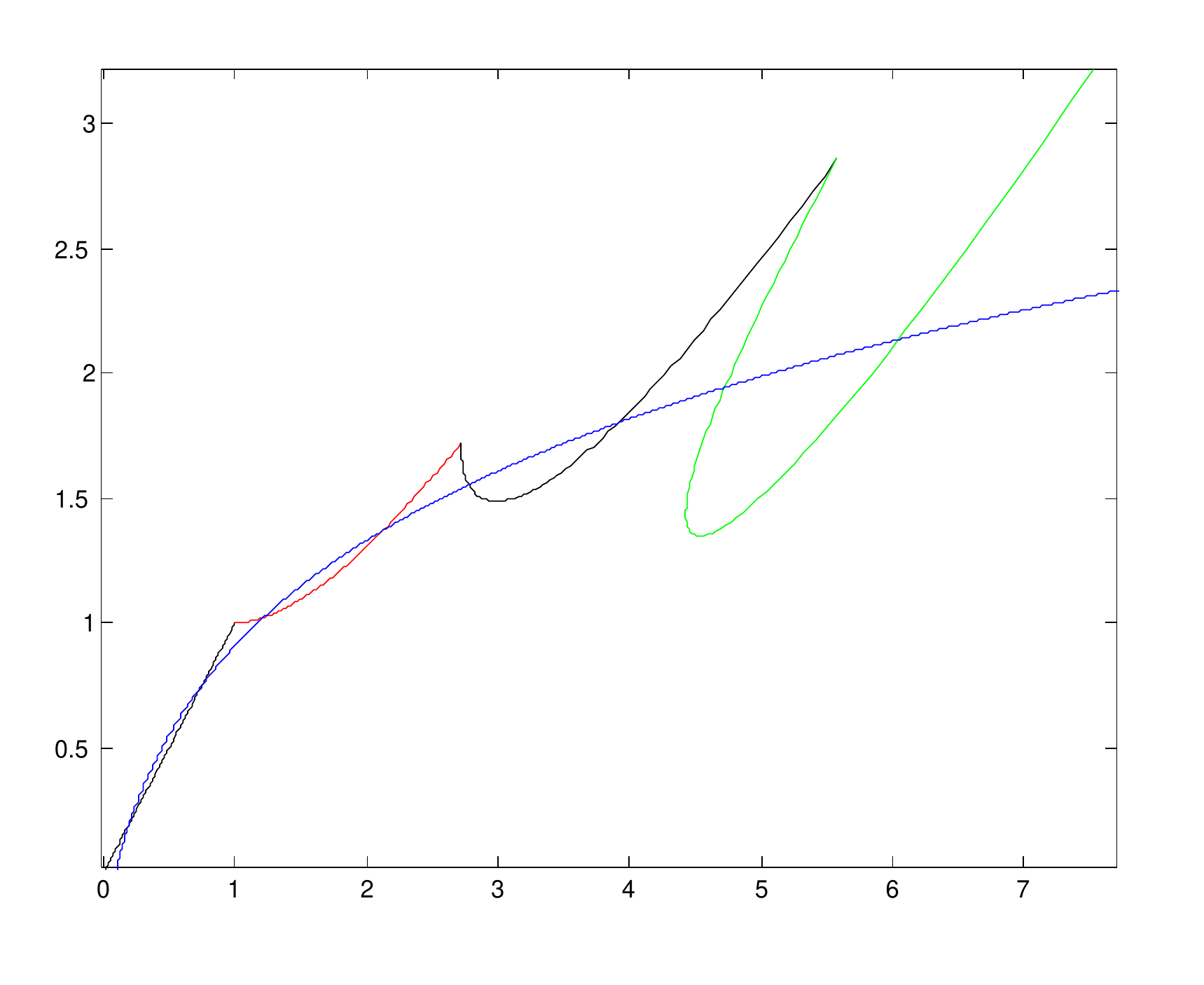}

\caption{Invariant curve and iterated initial data in the exponential case in $(y,z)$ coordinates. The long curve is the separatrix. It corresponds to the 
solid curve in Figure \ref{fig:expo}. 
The line segment with the end points $(0,0)$ and $(1,1)$ represents 
a one parameter family of the initial turning points $x_1$. Note that 
the segment intersects the separatrix at exactly two points. These two points 
are the candidates for the optimal search sequence.
The other curves are obtained by iterating the initial segment by the forward 
map.
}
\label{fig:exp_inv}
\end{figure}

\subsection{Existence of separatrix: exponential distribution}

The existence of the separatrix in the phase space for the exponentially
distributed $\HH$ is proved by applying  the standard Banach contraction 
mapping principle. 

We start by introducing more convenient coordinates in 
the phase space\footnote{Recall that $(x,y)$
represent the successive points of the trajectory $(x_k,x_{k+1})$.}  
$(x,y)\rightarrow (y,z=y-x)$. 
Thus, $z_{k+1} = x_{k+1}-x_k$ ``measures'' monotonicity of the orbits.

In these new coordinates, the mapping
$\vr$ is given by
\begin{equation}\label{eq:yz}
\vr:(y,z)\mapsto (Y,Z)=(\exp{z}, \exp{z} -y).
\end{equation}
The inverse map in these coordinates acts as
\begin{equation}\label{eq:yz-inv}
\vrinv: (Y,Z)\mapsto (y,z)= (Y-Z, \ln{Y}).
\end{equation}

The iterations of the boundary of monotonicity region $\{Z>0\}$ result
in curves $z=\phi_k(y)$, where the functions $\phi_k$ satisfy the
recursion
$$
\phi_{k+1}(Y-\phi_k(Y))=\ln(Y),
$$
or, equivalently,
$$
\phi_{k+1}(\eta)=\ln(\psi_k(\eta)),
$$
where $\psi_k$ is defined as inverse to $Y\mapsto Y-\phi_k(Y)$.

\begin{prp}\label{prp:separ}
The map $\phi_k\mapsto \phi_{k+1}$ defined above
is a contraction in the space of continuously 
differentiable positive functions with bounded derivative $0< \phi^{\prime}(y)<1/2$
for $y \geq 4$. There is a continuous limit $\phi=\lim_{k\to\infty} \phi_k$, 
which  solves the functional equation 
$$
\phi(y-\phi(y))=\ln(y)
$$
and satisfies the bound $|\phi(y)-\ln(y)|\leq 1$ on $y\in [4,\infty)$.
\end{prp}

By construction, the region below the separatrix $\separ$ (in $(y,z)$
coordinates) corresponds to the
non-monotonic solutions of the variational recursion, and that above 
$\separ$ correspond to monotonically increasing solutions. In other
words, $\separ$ is indeed the boundary of $\monot_\infty$.

\begin{proof}
Consider the inverse map \eqref{eq:yz-inv}.
It takes a graph $(y,\phi(y))$ into
a graph $(y,\Phi(y))$, where
\[
\Phi(\phi)(y) = \ln(w_{\phi}(y)),
\]
where $w_{\phi}(y)$ solves the equation
\[
y = w_{\phi}(y) - \phi(w_{\phi}(y)).
\]
We consider this mapping in the space of continuously differentiable functions
\[
{\bf X} = \{ \phi \in C^1 (y_0,\infty),  \phi (y)>0,
0< \phi^{\prime}(y) \leq 1/2 \}.
\]
Note, that at each iteration we have a well defined function $w=w_{\phi}(y)$ 
and that $w_{\phi}(y) > y$. Indeed, by the implicit function theorem, 
we need $\phi^{\prime}(w) \neq 1$,
which we have since $\phi^{\prime}(y) \leq 1/2$ and $w_{\phi}(y) > y$.

First, show that we can iterate indefinitely:
\[
\Phi(\phi)(y) = \ln(w_{\phi}(y)) > \ln (y) > \ln (y_0) > 0,
\]
if $y_0>1.$
Differentiating
\begin{equation}\label{eq:aprider}
\frac{d}{dy}\Phi(\phi)(y) = \frac{w_{\phi}^{\prime}(y)}{w_{\phi}(y)} = 
\frac{1}{w_{\phi}(y)} \cdot  \frac{1}{1-\phi^{\prime}(w_{\phi}(y))} \leq 
\frac{2}{w_{\phi}(y)} \leq \frac{2}{y}
\leq \frac{2}{y_0} \leq \frac 12,
\end{equation}
if $y_0>4$. Also, since $w_{\phi}^{\prime}(y) > 0$, we have 
\[
\frac{d}{dy}\Phi(\phi)(y)>0.
\]

Now, we show that the mapping $\Phi$ is a contraction 
in the space of continuous functions. 
Let $y \geq y_0$ and consider
\begin{align}\label{eq:phidiff}
|\Phi(\phi)(y) - \Phi(\psi)(y)| = |\ln(w_{\phi}(y))-\ln(w_{\psi}(y))| 
\end{align}
\[
\leq \frac{1}{\min  (w_{\phi}(y), w_{\psi}(y) )}\cdot |w_{\phi}(y)- w_{\psi}(y)| \leq
\frac{1}{y_0} |w_{\phi}(y)- w_{\psi}(y)|.
\]
Now, observe that
\[
 |w_{\phi}(y)- w_{\psi}(y)| = | \phi(w_{\phi}(y))- \psi(w_{\psi}(y))  | \leq
 |\phi(w_{\phi}(y))-\phi(w_{\psi}(y))| + |\phi(w_{\psi}(y))-\psi(w_{\psi}(y))|
\]
\[
\leq
\sup_{y\geq y_0} |\phi^{\prime}| \cdot | |w_{\phi}(y)- w_{\psi}(y)| + 
\sup_{y\geq y_0} |\phi(y)-\psi(y)|.
\]
Therefore,
\[
| w_{\phi}(y)- w_{\psi}(y) | \leq
\frac{\sup_{y\geq y_0} |\phi(y)-\psi(y)|}{1- \sup_{y\geq y_0} |\phi^{\prime}|} 
\leq 2\sup_{y\geq y_0} |\phi(y)-\psi(y)|
\]
and combining this inequality with \eqref{eq:phidiff}, we obtain
the contraction
\[
\sup_{y\geq y_0} |\Phi(\phi)(y) - \Phi(\psi)(y)| \leq
\frac{2}{y_0} \sup_{y\geq y_0} |\phi(y)-\psi(y)| \leq
\frac{1}{2} \sup_{y\geq y_0} |\phi(y)-\psi(y)|,
\]
assuming again that $y_0>4$.

As usual, in the contraction argument, the distance between initial guess
$\phi_0(y) = \ln(y)$ and the limit $\phi(y)$ is bounded by
$||\phi-\phi_0|| \leq 2||\phi_1-\phi_0||$. Consider
\[
|\phi_1(y)-\phi_0(y)| = | \ln(t(y)) - \ln(y)|,
\]
where $y=t(y)-\ln(t(y))$ with $y\geq 4$. Thus,
\[
 | \ln(t(y)) - \ln(y)| \leq \frac{1}{y} |t(y)-y| \leq 
\frac{1}{y} |t^{\prime}(y)-1|\cdot y =  |t^{\prime}(y)-1| =
\frac{1}{|t(y)-1|} ,
\]
where we used the derivative of the inverse function.
Since, we assume that $y\geq 4$ which implies  then  $t(y) > 2$, we have
\[
|\phi(y)-\phi_0(y)| \leq 2|\phi_1(y)-\phi_0(y)| \leq 1. 
\]

\end{proof}

Now, we verify that the obtained separatrix is actually smooth. We need this
property as we later prove that the cost function increases away from 
the separatrix.
In fact, the separatrix is possibly an analytic function, see the appendix. 

\begin{prp}
The separatrix is a continuously
differentiable function on the interval $[13,\infty)$ satisfying  the bound
\[
\frac{d}{dy} \phi(y) \leq \frac{2}{y}.
\]
\end{prp}
\begin{proof}
Now we consider contraction in the space of continuously differentiable
functions with the norm 
\[
||\phi||_1:=\sup_{y\geq y_0} |\phi| + \sup_{y\geq y_0} |\phi^{\prime}|
\]
and with the bound
\[
|\phi^{\prime \prime}(y)| \leq 1.
\]
We will also use the notation
\[
||\phi||_0:=\sup_{y\geq y_0} |\phi|.
\]

Using the definition of $\Phi(\phi)$ and of $w_{\phi}$, we calculate
\[
\Phi^{\prime \prime}(\phi)(y) = \frac{w^{\prime \prime}_{\phi}}{w_{\phi}} - 
\frac{(w^{\prime}_{\phi})^2}{w^2_{\phi}}
\]
and
\[
w_{\phi}^{\prime \prime}= \frac{\phi^{\prime \prime}(w_{\phi})(w^{\prime}_{\phi})^2  }{1-\phi^{\prime}(w_{\phi})}.
\]
Recalling that for $y_0 \geq 4$, we have $0<\phi^{\prime}(y)<1/2$ and 
$1< w_{\phi}^{\prime}(y) <2$ so that
\[
|w_{\phi}^{\prime \prime}(y)| \leq 8 |\phi^{\prime \prime}(y)| \leq 8.
\]
Next, we have
\[
|\Phi^{\prime \prime}(\phi)(y)| \leq \frac{|w_{\phi}^{\prime \prime}(y)| }{y_0} + \frac{4}{y_0^2}.
\]
Taking, {\rm e.g.}  $y_0=10$, we can ensure that the last expression is bounded by 1.

Now, we prove that we indeed have contraction 
\[
||\Phi(\phi(y)) - \Phi(\psi(y))||_1 = 
\sup_{y\geq y_0} | \Phi(\phi(y)) - \Phi(\psi(y))| + 
\sup_{y\geq y_0} |\Phi^{\prime}(\phi(y)) - \Phi^{\prime}(\psi(y)) |.
\]
We already know that
\[
\sup_{y\geq y_0} | \Phi(\phi(y)) - \Phi(\psi(y))| \leq 
\frac{2}{y_0}|\phi-\psi|_0 \leq \frac{2}{y_0}|\phi-\psi|_1.
\]

Now, we estimate
\[
|\Phi^{\prime}(\phi(y)) - \Phi^{\prime}(\psi(y))| = \left |  
\frac{w_{\phi}^{\prime}}{w_{\phi}} - \frac{w_{\psi}^{\prime}}{w_{\psi}} \right | \leq
\frac{ |w_{\psi}|\cdot |  w_{\phi}^{\prime}  - w_{\psi}^{\prime} | + 
|w_{\phi}^{\prime}|\cdot |w_{\psi}-w_{\phi}|   }{|w_{\phi}||w_{\psi}|}. 
\]

Using the estimates obtained in the proof of Proposition \ref{prp:separ}, we have
\[
\frac{|w_{\phi}^{\prime}| }{|w_{\phi}||w_{\psi}|}  \leq \frac{2}{y_0^2}
\]
and
\[
|w_{\phi}(y)-w_{\psi}(y)| \leq 2 \sup_{y\geq y_0} |\phi(y) -\psi(y)|.
\]
On the other hand, differentiating the identity 
\[
y = w_{\phi}(y) - \phi(w_{\phi}(y))
\]
and using triangle inequalities, we can estimate the difference 
\[
| w_{\phi}^{\prime}  - w_{\psi}^{\prime} |  \leq 
|\phi^{\prime}(w_{\phi})|\cdot |w_{\phi}^{\prime}  - w_{\psi}^{\prime}| + 
|w_{\psi}^{\prime}| ( |\phi^{\prime}(w_{\phi}) -  \psi^{\prime}(w_{\phi}|  +  
|    \psi^{\prime}(w_{\phi})-\psi^{\prime}(w_{\psi})  |  ).
\]
The first difference on the right hand-side can be absorbed into the left hand-side as we did in the proof of Proposition \ref{prp:separ}. 
The second difference is estimated by
\[
| \phi^{\prime}(w_{\phi}) -  \psi^{\prime}(w_{\phi}) | \leq ||\phi-\psi||_1
\]
and the third one,
\[
|\psi^{\prime}(w_{\phi}) - \psi^{\prime}(w_{\psi})| \leq ||\psi^{''}||_0\cdot |w_{\phi}-w_{\psi}| \leq |w_{\phi}-w_{\psi}|,
\]
where $|w_{\phi}-w_{\psi}|$ has been estimated in Proposition \ref{prp:separ}.

Combining these inequalities, we obtain
\[
|\Phi^{\prime}(\phi(y)) - \Phi^{\prime}(\psi(y))| \leq 
   \left ( \frac{12}{y_0} + \frac{4}{y_0^2} \right ) 
||\phi-\psi||_1.
\]
By taking sufficiently large $y_0$, e.g. $y_0=13$ we obtain contraction in $C^1$.
Having established continuous differentiability of $\phi$, the  bound follows from the apriori estimate  \eqref{eq:aprider}.

\end{proof}

\begin{rem}
By iterating the inverse map, one can show that the separatrix is smooth on a 
larger  interval $[1,\infty)$.
\end{rem}

\subsection{Properties of the separatrix}
\begin{itemize}

\item
By construction, the region below the separatrix $\separ$ (in $(y,z)$
coordinates) corresponds to the
non-monotonic solutions of the variational recursion, and that above 
$\separ$ corresponds to monotonically increasing solutions. In other
words, $\separ$ is indeed the boundary of $\monot_\infty$.

\item
Using functional equation, it is possible to obtain logarithmic 
series expansion of the function $\phi$ defining the
separatrix near
$y=\infty$ (the derivation can be found in the appendix):
\[
\phi(y)=\ln(y)+\frac{\ln(y)}{y}+\ldots
\] 

\item In the standard coordinates, it is instructive to consider the separatrix as the stable
  invariant manifold of a topological saddle ``at infinity''. The
  intuition behind this picture underlies the construction of the
  separatrix.
\end{itemize}

\section{Cost function and optimal trajectories}
\label{sec-4}

To understand the properties of the cost function and its
approximations $E^N(x)$ we will need a standard  trick from hyperbolic dynamics.
There it is used to find fragile objects (invariant foliations) from robust ones
(invariant cones), see e.g. \cite{katok}.

\subsection{Consistent cone fields }
\label{cons_fields}
We will continue to work in $(y,z)$ coordinates.

We will refer to a pair of nowhere collinear vector fields
$(\eta(y,z),\xi(y,z))$ (or, rather, to the convex cone in the tangent spaces
spanned by these vector fields)
as the {\em cone field} $K_{(y,z)}$, and to the vector fields $\eta,\xi$ as the
{\em generators} of $K_{(y,z)}$. We will say that the cone field $K_{(y,z)}$ is {\em
  consistent} at $(y,z)$, if the variational recursion {\bf R} maps it into itself, i.e.
$$
D\vr K_{(y,z)}\subset K_{\vr(y,z)};
$$
here $D\vr$ is the differential of $\vr$. For exponential $\HH$, it is 
given in the coordinates
$(y,z)$ by \\

\[
D \vr (y,z) = \ba{0} & {e^{z}}\\{-1} & {e^z}\ea \\
\]

\vspace{3mm}

We will call a subset $A$ of the quadrangle  $\{y\geq 0, z\geq 0\}$
a {\em $\vr$-stable set} if it is mapped into itself, {\em i.e.} $\vr (A) \subset A$.

\begin{prp}
The subset of the quadrangle ${\bf A} = \{y \geq 0, z \geq \max(0,\phi(y))\}$ is a 
{\bf R}-stable set.
\end{prp}

In other words, all the points in the positive quadrangle and above the separatrix do not leave that region under the action of {\bf R}. This statement follows 
from invariance of the separatrix and that the ray $\{y=0,z\geq 0\}$ and 
the segment $\{0\leq y\leq y^*,z=0\}$ are mapped inside ${\bf A}$, where $(y^*,0)$
is the point where the separatrix intersects $y$-axis.

Now we will construct an explicit consistent cone field for the
exponential  \HH. It is in fact just the constant field, spanned by the
tangent vectors $\eta=(1,2)$ and $\xi=(2,1)$

A straightforward  computation
shows that in the region $\{z>\ln 4\}$ the cone field generated by
$\eta$ and $\xi$ is consistent, and we deduce

\begin{prp}
In the region $z\geq \ln 4$ above the separatrix, which is a $\vr$-stable set there exists a consistent cone field transversal to the vertical vector
field $(0,1)$.
\end{prp}

\subsection{Monotonicity of the cost function on intervals of regularity}
\label{sec-4.1}

Now we are ready to prove the key fact about the cost function $E(x)$.
Consider the ray $\rr:=\{(t,t), 0<t<\infty\}$ of initial conditions
for the variational recursion. We will
say that $t_*$ is a regular point, if some vicinity of $t_*$ in the
ray $\rr$ belongs to the monotone region $\monot_\infty$. In other
words, for the initial data $x_0=0, x_1=t$, where $t$ is close to
$t_*$, the variational recursion generates an increasing trajectory,
for which the cost function is a well defined function $E(x)$.

It turns out that $x_*$ {\em cannot be a local extremum} of $E(x)$.

\begin{prp}\label{prp:momotonicity}
In $(y,z)$ coordinates,   if the region above the separatrix supports a consistent
  cone field $K$, with $\eta$ being one of the generators, and $\eta$  is not $\vr$-invariant then 
on any interval $I=(y_-, y_+)\subset \rr$ in the intersection of the ray of
  initial data with the monotone region  $\monot_\infty$ the function $E(x)$ is monotone. 
\end{prp}

\begin{proof}
Consider partial sums $E^N(x)$ which approximate $E(x)$:

\begin{equation}
  \label{eq:approx}
  E^N(x)=\sum_{m=0}^N f(x_m)x_{m+1}, 
\end{equation}
where the trajectory $\xx_\vr(x)$ solves the variational recursion. It is immediate that $E^N(x)$ is a smooth function of $x$, if $f(x)$ is.

As $E^N(x)$ converge pointwise to $E(x)$, non-monotonicity of $E$ on $I$ would
imply that for some compact subinterval $J\subset I$,
all the functions $E^N$ have a critical point on $J$ provided $N$ is sufficiently large.

By (\ref{eq:diff_CK}), 
$$
\frac{dE^N}{dx}=f(x_N)\frac{dx_{N+1}}{dx},
$$
and criticality 
$\frac{dE^N}{dx}=0$ is possible only if ${dx_{N+1}}/{dx}=0$ at some
point $x_*$ of $J$. 

As the $N$-th iteration of the initial point
$(y,z)=(x_1, x_1-x_0)$ is $(x_{N+1}, x_{N+1}-x_N)$, the vanishing of
${dx_{N+1}}/{dx}=0$
means that in $(y,z)$ coordinates the $N$-th iteration by $D\vr$ of the tangent vector to the ray $\rr$ {\em
 is vertical}.

However, the line of the initial conditions is the diagonal $(y=t,z=t)$. Computer simulations, see Figure \ref{fig:exp_inv}, show that after several iterates, the ray gets mapped into the cone field (above $z=\ln 4$).

 As the $K$ is consistent above the
separatrix, the iterations of these tangent vectors under $D\vr$ will
still be in the interior of $K$, while the vertical vector field is
the generator of $K$. Hence, ${dx_{K+1}}/{dx}$ cannot vanish on $J$,
ensuring that vicinity cannot contain a local extremum of $E$.
\end{proof}

Therefore, the cost function can only achieve minimum at one of the 
points of intersection of the separatrix with the line of initial conditions.

\subsection{Simulations and optimal trajectories}
\label{sec-4.2}
In this section we present results of numerical computation of 
the cost function for the one-sided search problem. We also explain how our theory fits with these observations.

Figure \ref{fig:cost}  shows the plots of the cost of the trajectories $\xx_\vr$
for the exponentially distributed $\HH$, evaluated at both chaotic and monotone
trajectories. The simulation was stopped either when the trajectories
increased beyond some large threshold, or after a fixed number of
steps (the former trigger would correspond to monotone trajectories;
the latter to chaotic ones).
\begin{figure}[htb]
  \centering

  \includegraphics[height=2.5in, width=.48\textwidth]{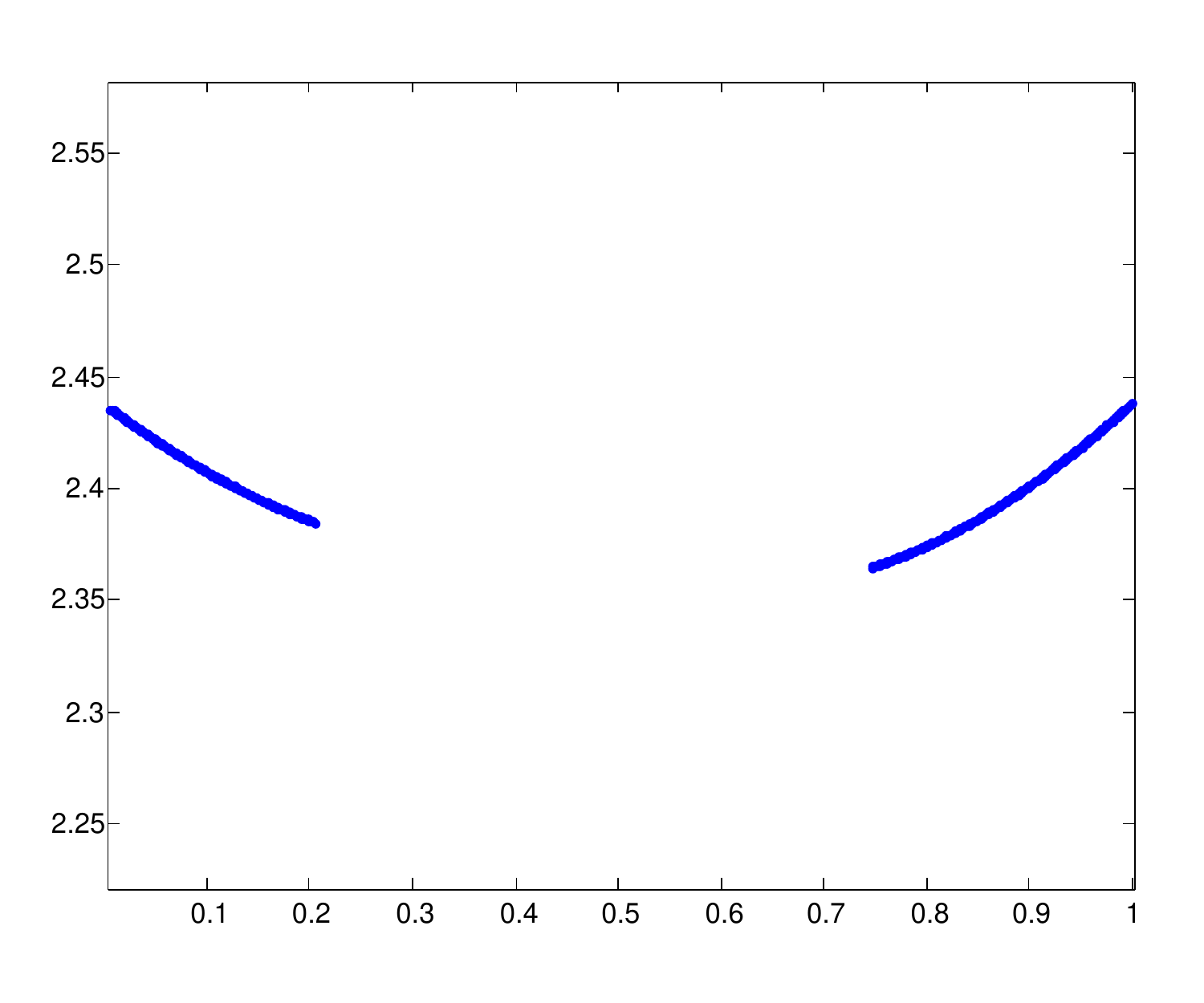}
  \includegraphics[height=2.5in, width=.48\textwidth]{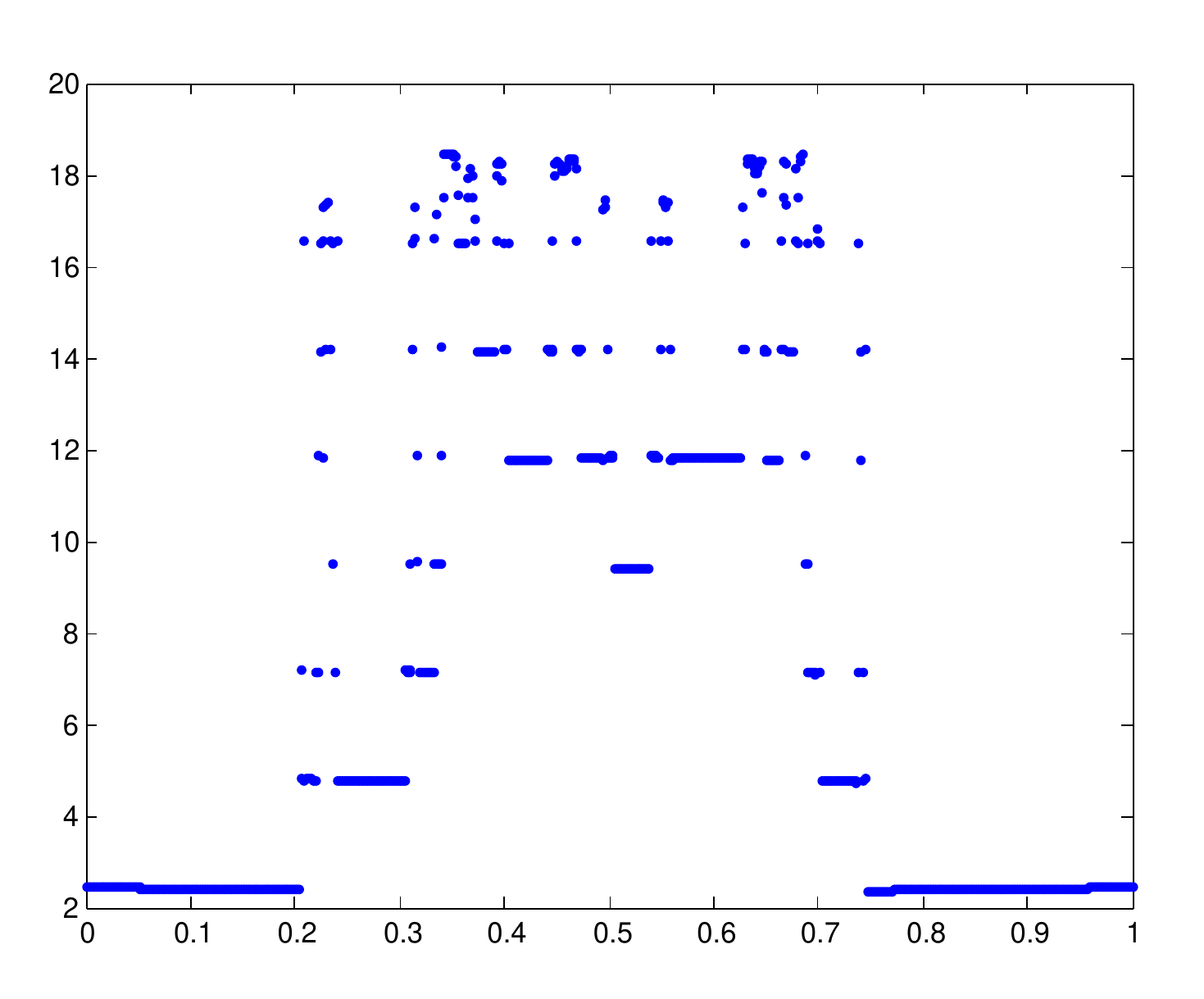}
  \caption{Numerically evaluated cost function $E(x)$ for exponentially distributed 
    $\HH$. Right display shows also results for chaotic region (stopped
    after a fixed number of iterations). Left display is a
    magnification of the right one, showing only the results over the
     region of monotonicity.}
  \label{fig:cost}
\end{figure}

The monotonicity of the cost over the left and right  intervals is apparent.
The separatrix $\separ$ intersects the ray of
initial conditions $\rr$ at two points, $x_+\approx 0.7465...$ and
$x_-\approx .1954...$
(compare with Figure \ref{fig:expo}). Between the points, the
initial conditions are in the chaotic region. The monotonicity of $E$
outside of the chaotic region means that one of the two initial
values, $x_+$ or $x_-$ should generate the optimal
trajectory. Numerically, $x_+$ wins: $E(x_+)\approx  2.3645<E(x_-)\approx 2.3861$.

\section{Conclusion}
\label{sec-5}
We developed a geometric approach to the Linear Search Problem via discrete time
Hamiltonian dynamics, which explains some of the hidden structure of the cost function.
The rapid decay of the tail distribution function translates into  
hyperbolicity of the underlying Hamiltonian dynamics. The latter is  defined 
by the variational recursion which plays 
a key role in
the characteristics of the optimal search trajectory. In particular, 
hyperbolicity implies the existence of separatrix which divides
the regular and chaotic regions, and the optimal search trajectory
needs to start on the separatrix: the chaotic region cannot contain optimal 
orbits, while in the regular region the orbits father away from separatrix 
have higher cost (monotonicity of the cost function). 

While this scenario is proved in this note only for a specific case of
exponential tail distribution function, we anticipate that 
for other distributions with sufficiently fast decay, the same
type of results, including the existence of separatrix and
monotonicity of cost function in the region of monotonicity, will hold.
Some of this hope is supported by partial results, see the appendix.

We plan to return to this more general classes of distributions in a
follow-up paper, where we also plan to address the phenomenon of
separatrix slow-down (the growth of trajectories on separatrix is
slower than that in the interior of the  region of monotonicity).

There are other open questions arising in the context of  Hamiltonian 
dynamics based approach to the search problem. Extending the
set of analyzed distributions to those with bounded
support is a natural task.

We also expect that 
in  the search on rays, where the corresponding Hamiltonian map is higher
 dimensional, hyperbolicity will also play an important role and
 higher dimensional separatrix (unstable manifold) can be found. 
It is expected that optimal search plan would still be restricted
 to the unstable manifold.

\pagebreak

\appendix

\section{Series expansions}

The expansion near $x=\infty$ for the separatrix given by 
\[
\phi(x-\phi(x)) = \ln(x)
\]
leads to logarithmic series 
\[
\phi(x) = \sum_{n=0}^{\infty} \frac{Q_n(\ln(x))}{x^n}.
\]
The first three terms are given by 
\[
\phi(x) = \ln (x) + \frac{\ln(x)}{x}+ \frac{1}{x^2} 
\left (
\frac 12 + \frac 34 \ln(x) -\frac 13 \ln^2(x) 
\right )+ ...
\]

To justify this expansion, we need 
\begin{lem}
The equation $x = t(x)-\ln t(x)$ has a smooth solution for sufficiently
large $x$
\[
t(x) = x + \ln x + O \left (\frac{\ln x}{x} \right ).
\]
\end{lem}
\begin{proof}
Let us write 
\[
t(x) = x + \ln x + r(x)
\]
and substitute in the equation. After some simplifications, we have
\[
r = \ln \left ( 1+\frac{\ln x}{x} +\frac{r(x)}{x}   \right ).
\]
Application of the contraction mapping principle  to $r(x)$ gives the required error estimate.
\end{proof}

Now, we prove
\begin{prp}
\[
\phi(x) = \ln x +O \left ( \frac{\ln x}{x} \right ).
\]
\end{prp}
\begin{proof}
Consider the first two iterations by ${\bf R}^{-1}$ of $\phi_0:=(x=t, y=0)$,
\[
\phi_1:=(x=t, y=\ln t), \phi_2:= (x= t-\ln t, y = \ln t).
\]
They can be represented as graphs $y=\phi_1(x),y=\phi_2(x)$ for 
sufficiently large $x$. Note that $\phi_1(x) = \ln (x)$, while 
$\phi_2(x) = \ln t(x)$, where $x=t(x)-\ln t(x)$.

Now, using the above lemma we estimate  
\[
|\phi_2(x)-\phi_1(x)| = |\ln t(x) - \ln x| = 
|\ln \left ( x+\ln x + r(x)   \right ) -\ln x|=
\left | \ln \left ( 1+\frac{\ln x}{x} +\frac{r(x)}{x}   \right ) \right | 
\leq C \, \frac{\ln x}{x}
\]
Applying contraction mapping principle, we obtain the desired estimate 
\[
|\phi(x) - \ln x| \leq C \, \frac{\ln x}{x}.
\]
\end{proof}

\begin{thm}
The mapping {\bf R} restricted to the separatrix takes the form
\[
x_{n+1} = x_n + \ln(x_n) + O(\ln(x_n)/x_n)
\] 
\end{thm}
\begin{proof}
The separatrix is given by 
\[
\phi(x) = \ln(x) + \rho \left ( \frac{\ln(x)}{x} \right )
\]
for $x\rightarrow \infty$.

Then, using the forward map representation $(x_{n+1},y_{n+1})= (\exp y_n, x_{n+1}-x_n)$, we have
\[
\ln(x_{n+1}) + \rho (x_{n+1}) = x_{n+1}-x_n, 
\]
where $\rho(x) = O(\ln x/x)$ is a smooth function. 
Applying the implicit function theorem and estimating the error term, we obtain the result.

\end{proof}

\begin{thm}
The asymptotics of the mapping restricted to the separatrix is given by
\[
x_n = n(\ln(n) + \ln(\ln (n)))+r_n,
\]
where $r_n$ is a sequence satisfying
\[
|r_{n+1}-r_n| \leq C. 
\]
\end{thm}
\begin{proof}
Substitute the expansion of $x_n$ in the recurrent relation 
\[
x_{n+1} = x_n + \ln(x_n) + O(\ln(x_n)/x_n),
\] 
then after some cancellations, we obtain that $r_{n+1}=r_n +1 + O(1)$ which 
implies the result.
\end{proof}

\section{Two-sided Gaussian distribution: Beck-Bellman problem}
We consider the two-sided search on the real line with
Gaussian probability distribution function as in the original 
Beck-Bellman problem and we show numerically that the same canonical structure 
persists: separatrix intersecting the curve of initial turning points.

\begin{figure}

\includegraphics[width=100mm]{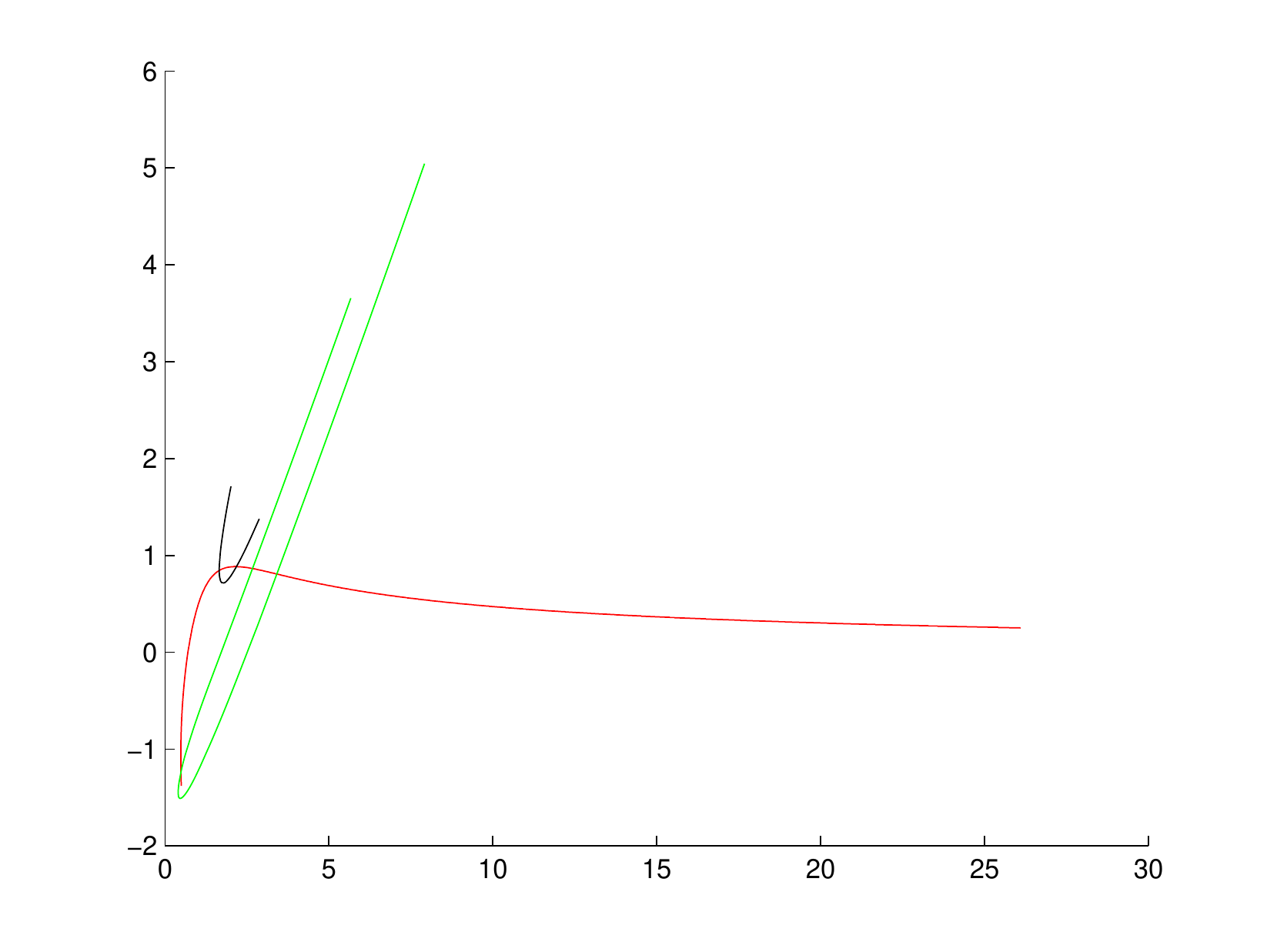}

\caption{Invariant curve and iterated initial data in Beck's problem. 
The long curve
is invariant manifold. The other two bended curves are 1st and 2nd forward iterates of initial data. The initial data itself is not present because continuation of the separatrix in that region is computationally too difficult.}
\label{2pointsbeck}
\end{figure}

The difference relation obtained in \cite{beck}, is given by
\[
(x_n + x_{n+1})\phi(x_n) = G(x_n)+G(x_{n-1}),
\]
where
\[
\phi(t) = \frac{1}{\sqrt{2\pi}}e^{-t^2/2}, \,\,\,\,
G(x) = \int_x^{\infty}\phi(t) dt.
\]
The actual turning points are $(-1)^n x_n,$ while $x_n \geq 0$.
For matlab computations, we use
\[
{\erfc(x)} = \frac{2}{\sqrt{\pi}}\int_x^{\infty} e^{-t^2} dt
\]
and the inverse function called ${\rm erfcinv}$.
Using the  relation
\[
G(x) = \frac{1}{2}\erfc (x/\sqrt{2}).
\]
the finite difference relation takes the form
\[
(x_{n+1}+x_{n})\phi(x_n) = \frac{1}{2}( \erfc (x_n/\sqrt 2)+\erfc (x_{n-1}/\sqrt 2)).
\]
Now, using $y_{n+1}=x_{n+1}-x_n$, we have
\begin{align}
x_{n+1}=\frac{1}{2\phi(x_n)} (  {\erfc} (x_n/\sqrt 2 )+{\erfc} ((x_n-y_n)/\sqrt 2) ) - x_n.
\end{align}

We will also use the inverse map which takes the form
\begin{align}
&x_{n+1} = x_n-y_n \nn \\
&y_{n+1} = x_{n+1}-\sqrt{2} \, {\rm erfcinv}\, ( 2\phi(x_{n+1})(x_n+x_{n+1}) - 
{\erfc}(x_{n+1}/\sqrt{2})). \nn
\end{align}
In this case, the initial data is given by the line segment $x_1=y_1 =t$.

\section{Gaussian tail distribution. One-sided search.}
In this section we verify that contraction mapping principle can be used to establish existence of 
separatrix for the one-sided 
search problem with Gaussian tail distribution.

In this case $f(x) =  e^{-x^2}$, so that the second order difference
 relation is given by
\[
x_{n+1}= \frac{1}{2x_n} \, e^{\, x_n^2-x_{n-1}^2}.
\]
Let $y_{n+1}=x_{n+1}^2-x_n^2$, then we have
\begin{align}
&x_{n+1} =\frac{1}{2x_n}e^{y_n} \nn \\
&y_{n+1} =x_{n+1}^2-x_n^2.    \nn
\end{align}
We will also need the inverse map
\begin{align}
&x_n = \sqrt{x_{n+1}^2-y_{n+1}} \nn \\
&y_n = \ln {(2x_n x_{n+1})}   \nn
\end{align}
In this case, the initial data is given by a quadratic curve
\[
y=x^2=t^2.
\]

Now, we show that the contraction principle  can be extended to Gaussian case.
\begin{theo}[Unstable invariant manifold for one-sided Gaussian]
There exists an invariant manifold containing a graph $y=h(x)$ on $x\in [x_0,\infty)$ and with
\[
|h(x)-\ln (2x^2)|< 1.
\]
\end{theo}
\begin{proof}
Set up contraction mapping
\[
\Phi(\phi)(x) = \ln(2 z_{\phi}(x) x),
\]
where
\[
z^2_{\phi}(x) -\phi(z_{\phi}(x)) = x^2.
\]
Let
\[
{\bf X} = \{ \phi \in C^1 (x_0,\infty),  \phi(x)>0, 0< \phi^{\prime}(x) \leq 1/2 \}.
\]
By applying the same argument as in the exponential case, we can ensure that $\Phi$ leaves ${\bf X}$ invariant if we
take as the initial guess $\phi_0(x) = \ln(2x^2)$.

To establish contraction, consider
\[
|\Phi(\phi)(x) - \Phi(\psi)(x)| = |\ln(2xz_{\phi}(x)) - \ln(2xz_{\psi}(x))| =
\]
\[
|\ln(z_{\phi}(x)) - \ln(z_{\psi}(x))| \leq \frac{1}{\min(z_{\phi}(x),z_{\psi}(x))} |z_{\phi}(x)-z_{\psi}(x)|.
\]

\begin{figure}
\includegraphics[width=100mm]{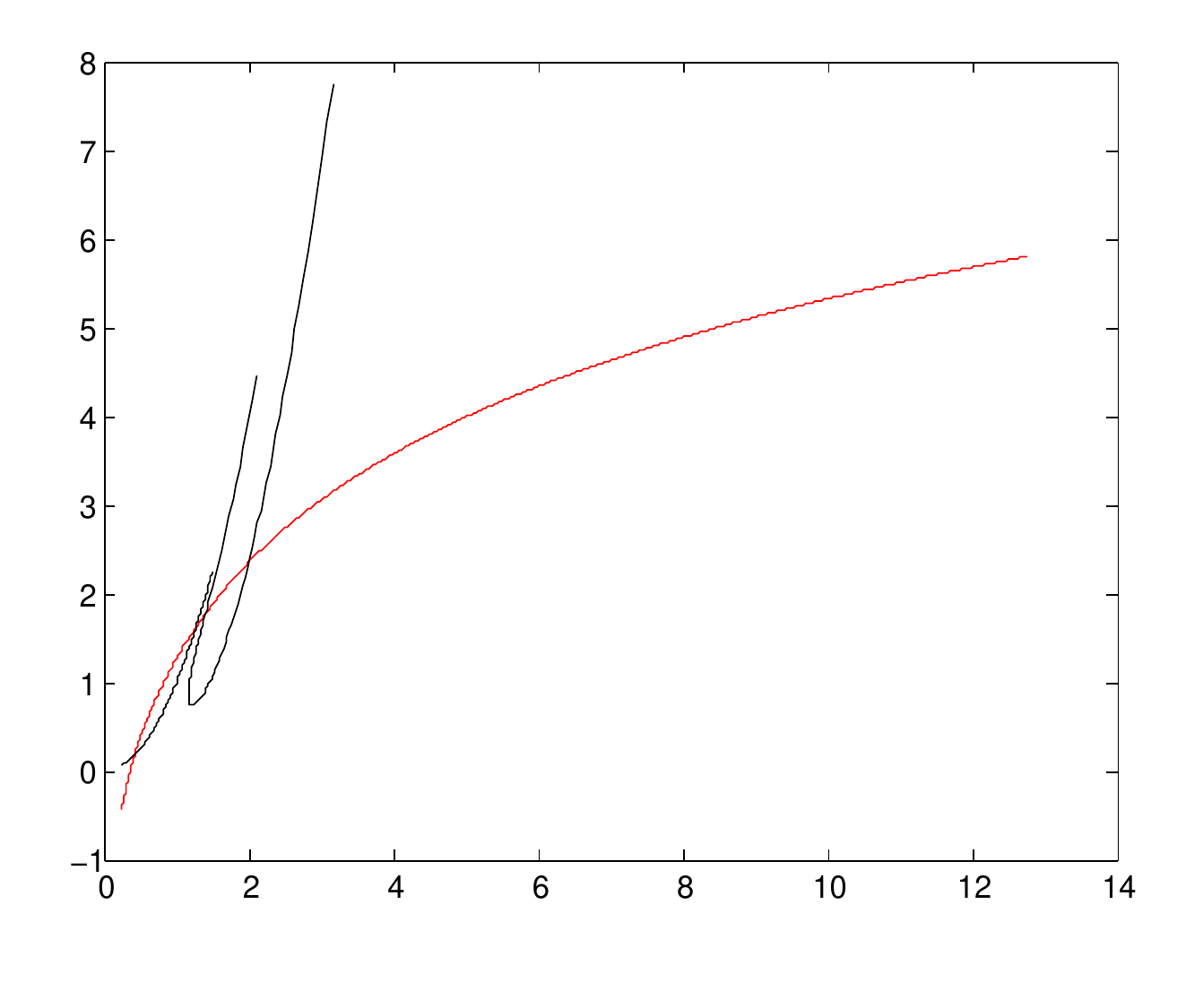}
\caption{Invariant curve and iterated initial data.The longer curve is 
the invariant manifold. Two other curves are iterated initial turning points.}
\label{2ptonesidegauss}
\end{figure}

Using the identity
\[
z_{\phi}^2(x) - z_{\psi}^2(x) = \phi(z_{\phi}(x)) - \psi(z_{\psi}(x)),
\]
and that $z_{\phi}(x) \geq x$,
we have
\[
|z_{\phi}(x)-z_{\psi}(x)| \leq \frac{1}{z_{\phi}(x)+z_{\psi}(x)}| \phi(z_{\phi}(x)) - \psi(z_{\psi}(x))|
\]
\[
\leq
\frac{1}{2x} ( |\phi(z_\phi(x)) - \psi(z_{\phi}(x))|+ |\psi(z_{\phi}(x)) - \psi(z_{\psi}(x))|
\]
\[
\leq \frac{1}{2x}
\left (
||f-g|| + ||g^{\prime}||\cdot |z_{\phi}(x)-z_{\psi}(x)|
\right ).
\]
Combining the terms, we have
\[
 |z_{\phi}(x)-z_{\psi}(x)| \leq \frac{1}{2x - ||{\psi}^{\prime}||}\, ||\phi-\psi||
\]
and then
\[
|\Phi(\phi)(x) - \Phi(\psi)(x)| \leq  \frac{1}{2x} \cdot \frac{1}{2x - ||\psi^{\prime}||} \, ||\phi-\psi||.
\]
Since we have assumed the bound $0<{\psi}^{\prime}<1/2$, taking $x\geq 1$, we obtain contraction
\[
|\Phi(\phi)(x) - \Phi(\psi)(x)| \leq  \frac{1}{3} \, ||\phi-\psi||.
\]
\end{proof}

\end{document}